\documentclass[a4paper,11pt]{article}
\usepackage{amsmath,amssymb,amsthm}
\usepackage{cite}
\usepackage{graphicx}

\theoremstyle{plain}
\newtheorem{theorem}{Theorem}
\newtheorem{lemma}{Lemma}
\newtheorem{proposition}{Proposition}
\theoremstyle{definition}
\newtheorem{definition}{Definition}

\begin{document}
\renewcommand{\thefootnote}{\fnsymbol{footnote}}

\author{M.~I.~Katsnelson\footnote{Theory of Condensed Matter,
Institute for Molecules and Materials, Radboud University Nijmegen,
The Netherlands. E-mail: m.katsnelson@science.ru.nl} and
V.~E.~Nazaikinskii\footnote{A.~Ishlinsky Institute for Problems in
Mechanics, Moscow, Russia; Moscow Institute of Physics and
Technology, Dolgoprudnyi, Moscow District, Russia. E-mail:
nazay@ipmnet.ru}}

\renewcommand{\thefootnote}{\arabic{footnote})}
\setcounter{footnote}{0}

\title{The Aharonov--Bohm effect for massless Dirac
fermions\\ and the spectral flow of Dirac type operators\\ with
classical boundary conditions}

\date{}

\maketitle

\begin{abstract}
We compute, in topological terms, the spectral flow of an arbitrary
family of self-adjoint Dirac type operators with classical (local)
boundary conditions on a compact Riemannian manifold with boundary
under the assumption that the initial and terminal operators of the
family are conjugate by a bundle automorphism. This result is used
to study conditions for the existence of nonzero spectral flow of a
family of self-adjoint Dirac type operators with local boundary
conditions in a two-dimensional domain with nontrivial topology.
Possible physical realizations of nonzero spectral flow are
discussed.

\end{abstract}

\section{Introduction}

Not to mention high-energy physics and quantum field theory, the
ideas of modern geometry and topology become increasingly important
in condensed matter physics \cite{1*,2*,3*,4*,5*,6*,Kats,8*,9*}. In
particular, the Atiyah--Singer index theorem \cite{AtSi0} explains
a topological protection of zero-energy Landau level and related
peculiarities of the quantum Hall effect in graphene
\cite{6*,Kats}. Topologically protected zero modes play an
essential role in the motion of vortices in superfluid helium-3
\cite{5*,10*}. The quantum Hall effect \cite{2*,3*} and topological
insulators \cite{8*,9*} are examples of the states of matter with
topological order parameter.

The Aharonov--Bohm effect (ABE) \cite{11*,12*} has actually
initiated this development. A magnetic flux localized in a region
completely unavailable for a quantum particle (e.g., surrounded by
infinitely high potential barrier) nevertheless affects its motion,
modifying the geometry of quantum space. A periodic dependence of
electron energy levels in a ring as a function of the magnetic flux
through the ring resulting in appearance of a persistent current
(see Ref. \cite{13*} and references therein) is a bright
manifestation of ABE. When the change of the magnetic flow is equal
to an integer number of the flux quanta, the energetic spectrum
should return to its initial state. Until recently, ABE was studied
mainly for usual nonrelativistic electrons described by the
Schr\"{o}dinger equation. After discovery of graphene, the ABE for
ultrarelativistic electrons described by Dirac equation with zero
mass has attracted attention \cite{14*,15*,16*}. From the
mathematical point of view, this is a much richer case. The Dirac
operator is not semibounded and hence its spectral flow \cite{APS3}
can be nonzero. It is worth noting that the nonzero spectral flow
of the Dirac operator has been discussed already in a context of
condensed matter physics. It results in additional forces (``Kopnin
forces'') acting on vortices in superfluid helium-3 \cite{5*,10*}.
Coming back to ABE, it means that the coincidence of the whole
energy spectra at the change of the magnetic flux at integer number
of the flux quanta does not necessarily mean the periodicity of
individual eigenenergies (like the shift $m \rightarrow m+1$
transforms $Z$ to itself). Nonzero spectral flow corresponds to a
physical situation when an adiabatically slowly varying magnetic
field leads to a production of electron--hole (or, in general,
particle--antiparticle) pairs from the vacuum: ``positron'' levels
cross zero-energy level transforming into electron ones. The vacuum
reconstruction effects were discussed in physics of superfluid
helium-3 \cite{5*} and in physics of graphene \cite{Kats} but
without any relation with ABE. Here we study conditions of
existence of nonzero spectral flow of a family of Dirac-like
self-adjoint operators with local boundary conditions in a domain
with nontrivial topology.

\medskip

In a bounded domain $X\subset\mathbb{R}^2$ with smooth boundary
$\partial X$ (see Fig.~\ref{fig01}), consider the boundary value
problem
\begin{equation}\label{ei1}
     \biggl(\begin{matrix}
      0 & -i \frac{\partial }{\partial x}-\frac{\partial }{\partial y} \\
      -i\frac{\partial }{\partial x} +\frac{\partial }{\partial y}& 0 \\
    \end{matrix}\biggr)\begin{pmatrix}
      u_1 \\
      u_2 \\
    \end{pmatrix}u=\begin{pmatrix}
      f_1 \\
      f_2 \\
    \end{pmatrix}\quad \text{in $X$} ,\quad (n_y-in_x)u_1=Bu_2\quad\text{on $\partial X$,}
\end{equation}
where $n_x$ and $n_y$ are the inward normal components and $B$ is a
nonvanishing real-valued function on the boundary. (Berry and
Mondragon~\cite{19*} were the first to consider boundary conditions
of this kind.)
\begin{figure}
\centering
  \includegraphics{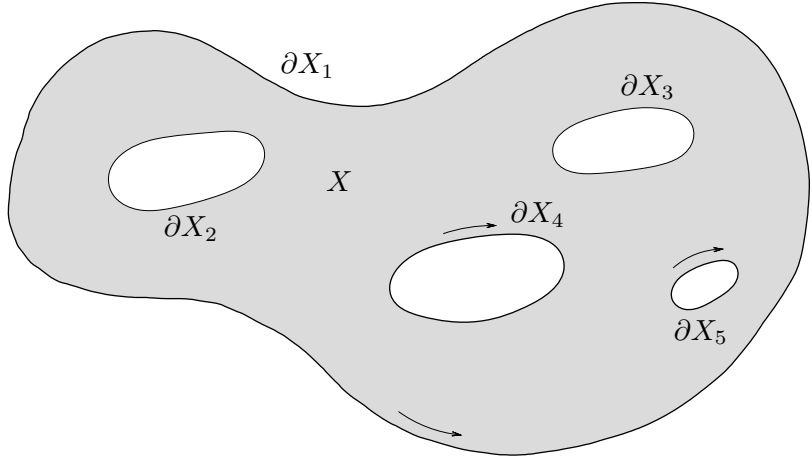}
 \vglue5pt
\caption{Example of a domain~$X$ with $m=5$ boundary components.
The bold lines show the boundary components $\partial^+\!X=\partial
X_1\cup\partial X_4\cup\partial X_5$ on which $B>0$ in the boundary
condition in~\eqref{ei1}. (It is only these components that
contribute to the spectral flow according to formula~\eqref{e9}.)
The arrows show the positive sense of going around the contour when
computing the winding number of the gauge transformation $\mu$
along the corresponding boundary component.\label{fig01}}
\end{figure}
The operator $D_0$ corresponding to this problem is self-adjoint
and Fredholm on $L^2(X,\mathbb{C}^2)$. Next, let $\mu$ be a smooth
function on $X$ with $\lvert\mu\rvert=1$. The ``gauge
transformation'' $D_0\to \mu D_0\mu^{-1}$ takes $D_0$ to the
operator $D_1=D_0+Q_1$ (where $Q_1$ is a self-adjoint matrix
function) with the same boundary conditions. In physicists'
language, the gradient of the phase $\mu$ is an abelian ($U(1)$)
gauge field, i.e., an electromagnetic vector potential. We consider
a two-dimensional domain with $m-1$ holes pierced by magnetic flux
tubes. The case in which all magnetic fluxes through the holes are
integer multiples of the magnetic flux quantum corresponds to
$\mu=1$. Let us join $Q_1$ with $Q_0=0$ by a continuous family
$Q_t$, $t\in[0,1]$, of self-adjoint matrix functions. The spectral
flow $\operatorname{sf}\{D_t\}$ of the family
\begin{equation}\label{ei1a}
  D_t=D_0+Q_t,
\end{equation}
i.e., the number of eigenvalues of $D_t$ that changed their sign
from minus to plus as the parameter $t$ varies from $0$ to $1$
minus the number of eigenvalues that changed their sign from plus
to minus, does not change under continuous deformations of the
family provided that $D_0$ and $D_1$ remain isospectral in the
course of deformation. As far as the authors know, the problem of
finding this spectral flow (also for families of Dirac operators of
more general form) was posed for the first time and partially
solved in \cite{Pro1}, where the spectral flow was computed up to
an integer factor $c_m$ depending on the number $m$ of boundary
components. Further, it was shown in \cite{Pro1} that $c_2=1$, and
it was conjectured that $c_m=1$ for all $m$. In the present paper,
we establish a general result (see Theorem~\ref{t1} below), which,
in particular, proves this conjecture to be true. Thus, Theorem 1
in \cite{Pro1} acquires the following form:
\begin{theorem}\label{co1}
The spectral flow of the family \eqref{ei1a} is given by the
formula
\begin{equation}\label{e9}
    \operatorname{sf}{D_t}=\operatornamewithlimits{wind}_{\partial\mathstrut^+\!X}\mu,
\end{equation}
where $\partial\mathstrut^+\!X$ is the part of $\partial X$ where
$B>0$ and
\begin{equation*}
    \operatornamewithlimits{wind}_{\partial\mathstrut^+\!X}\mu
    =\frac1{2\pi i}\oint_{\partial\mathstrut^+\!X}\frac{d\mu}\mu
\end{equation*}
is the winding number of the restriction of the function $\mu$ to
$\partial\mathstrut^+\!X$. \textup(The set
$\partial\mathstrut^+\!X$ is a union of finitely many
circles\textup; when defining the winding number, the positive
sense of any of these circles is the one for which the domain $X$
remains to the left when moving along the circle.\textup)
\end{theorem}
This theorem shows that the coefficients $c_m$ that remained
unfound in~\cite[Theorem~1]{Pro1} are equal to unity for all~$m$.
The same is true for \cite[Theorems~2 and~3]{Pro1}; all unknown
coefficients $c_m$ occurring there are equal to unity.

\medskip

Theorem~\ref{co1} follows from a general result established in the
present paper. We give a computation in topological terms of the
spectral flow of an arbitrary family $\{D_t\}$, $t\in[0,1]$, of
self-adjoint Dirac type operators with local boundary conditions on
a compact Riemannian manifold $X$ with boundary $\partial X$ under
the assumption that $D_1=U D_0 U^{-1}$, where $U$ is some
automorphism of the bundle in which the operators $D_t$ act. (In
contrast with~\cite{Pro1}, we assume neither that the principal
part of $D_t$ is independent of $t$ nor even that the principal parts of
$D_0$ and $D_1$ coincide.) Namely, we prove (see Theorem~\ref{t1}
below) that
\begin{equation}\label{ei2}
    \operatorname{sf}\{D_t\}=\operatorname{ind}\biggl(\frac{\partial }{\partial t} +D\biggr),
\end{equation}
where the right-hand side is the index of an elliptic operator with
boundary conditions on the manifold $X\times S^1$ with boundary,
$t$ being the coordinate on the circle $S^1$ and the operator $D$
being obtained from the family $\{D_t\}$ by clutching the operators
$D_0$ and $D_1$ with the use of the automorphism $U$. (Recall that
formula \eqref{ei2} for families of self-adjoint elliptic operators
on a \textit{closed} manifold $X$ was established in~\cite{APS3}.)
The right-hand side of \eqref{ei2} can be computed by the
Atiyah--Bott formula \cite{AtBo2} (see
also~\cite[Sec.~20.3]{Hor3}). Note, however, that we do not rely on
the Atiyah--Bott formula in the proof of Theorem~\ref{co1};
relation \eqref{ei2} between the spectral flow and the index
permits one to use the localization method
(see~\cite{R:NaSt3,R:NaSt7,AAA}) and cut the domain into parts,
thus reducing the problem to the case of a domain with one hole
($m=2$), for which a straightforward computation was carried out
in~\cite{Pro1}. Note also that the localization method proves to be
an important technical tool when proving relation~\eqref{ei2}
itself. The proof is in many aspects similar to that in
\cite[Proposition~5.6]{NSScS99} of the spectral flow formula for
families of differential operators Agranovich--Vishik elliptic with
parameter on a closed compact manifold but contains a number of new
important lines of argument related to the presence of boundary
conditions.

\section{Spectral Flow}

Recall the definition of spectral flow in the form presented
in~\cite{NSScS99} (cf.~\cite{BLP1}). Let $\{B_t\}$, $t\in[0,1]$, be
a family, continuous in the sense of uniform resolvent convergence,
of unbounded self-adjoint operators with purely discrete spectrum
on a Hilbert space $\mathcal{H}$. Then there exists a partition
$0=t_0<t_1<t_2<\dotsm<t_{n+1}=1$ of the interval $[0,1]$ and real
numbers $\gamma_1,\dotsc,\gamma_{n+1}$ such that $\gamma_j$ does
not lie in the spectrum $\operatorname{Spec}(B_t)$ of the operator
$B_t$ for $t\in[t_{j-1},t_j]$, $\gamma_1=\gamma_{n+1}\le 0$, and if
$\gamma_1<0$, then the half-open interval $[\gamma_1,0)$ does not
contain any points of spectrum of $B_0$ and $B_1$.
\begin{definition}[see~\cite{NSScS99}, Definition~A.18]\label{def-sf}
The \textit{spectral flow} of the family $\{B_t\}$, $t\in[0,1]$, is
the number\footnote{The right-hand side of formula~\eqref{e4} is
independent of the choice of the partition $\{t_j\}$ and the
numbers $\gamma_j$ by Theorem~A.19 in~\cite{NSScS99}.}
\begin{equation}\label{e4}
    \operatorname{sf}\{B_t\}
    =\sum_{j=1}^nm_j\operatorname{sign}(\gamma_j-\gamma_{j+1}),
\end{equation}
where $m_j$ is the number of eigenvalues \textup(counting
multiplicities\textup) of the operator $B_{t_j}$ on the interval
between $\gamma_j$ and $\gamma_{j+1}$.
\end{definition}
This definition is illustrated in Fig.~\ref{fig02}, which, in
particular, clarifies why this definition is consistent with the
notion of spectral flow as the number of eigenvalues passing
through zero (with direction taken into account).
\begin{figure}
\centering
  \includegraphics{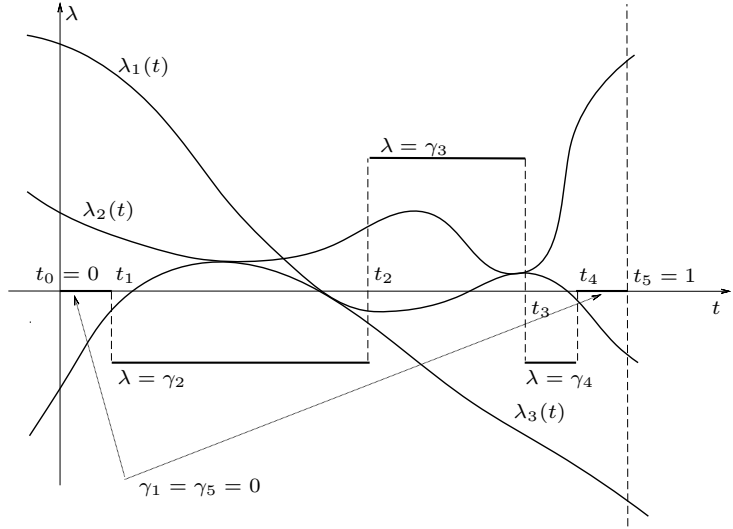}
 \vglue5pt
\caption{Definition of spectral flow. There are three eigenvalues
$\lambda_j(t)$, $j=1,2,3$, contributing to the spectral flow. Both
the computation according to the definition and counting the number
of zero crossings (with regard to direction) give the value $-1$
for the spectral flow.\label{fig02}}
\end{figure}

\section{Main Results}

Let $E$ be an even-dimensional Hermitian vector bundle over a compact
Riemannian manifold~$X$ with boundary, and let
\begin{equation}\label{e1}
  A\colon C^\infty(X,E)\longrightarrow C^\infty(X,E)
\end{equation}
be a formally self-adjoint Dirac type operator.\footnote{Recall
that a linear first-order differential operator \eqref{e1} is
called a Dirac type operator if its principal symbol
$\sigma_A(x,\xi)$ satisfies the condition
$(\sigma_A(x,\xi))^2\!=\!\sum g^{jk}(x)\xi_j\xi_k I$, where $I$ is
the identity operator in the fiber $E_x$ and the $g^{jk}(x)$ are
the (contravariant) components of the metric tensor
(see~\cite{BBW1}). The formal self-adjointness of $A$ is understood
in the standard sense as the condition that the identity
$(u,Av)=(Au,v)$ holds for any sections $u,v\in
C_0^\infty(X\setminus\partial X,E)$, where
$(\operatorname{\boldsymbol\cdot},\operatorname{\boldsymbol\cdot})$
is the inner product on $L^2(X,E)\equiv
L^2(X,E,d\operatorname{vol})$. Here $d\operatorname{vol}$ is the
Riemannian volume element on~$X$.} Next, let a subbundle $L\subset
E_Y$ of dimension $\dim L=\frac12\dim E$ be given in the
restriction $E_Y$ of the bundle $E$ to the boundary $Y=\partial X$
of the manifold $X$ such that
\begin{equation}\label{e2}
    \bigl(\sigma_{A}(x,\mathbf{n}(x))L_x\bigr)\perp L_x \qquad \forall x\in Y,
\end{equation}
where $L_x$ is the fiber of $L$ at $x$ and $\mathbf{n}(x)$ is the
unit inward conormal vector on the boundary. Consider the operator
\eqref{e1} on the set of sections $u\in C^\infty(X,E)$ satisfying
the homogeneous boundary condition
\begin{equation}\label{e3}
     \pi_L(u |_Y)=0,\quad\text{where }\pi_L\colon E_Y\longrightarrow E_Y\slash
     L\quad
     \text{is the natural projection.}
\end{equation}
(In other words, $u(x)\in L_x$ for $x\in Y$.) In particular, the
boundary condition in~\eqref{ei1} is of this form. It is well known
(see~\cite{BrLe3} and~\cite[Chaps.~18 and~19]{BBW1}) that the
boundary condition \eqref{e3} is elliptic, the operator \eqref{e1}
with domain given by this condition is essentially self-adjoint on
$L^2(X,E)$, and its closure~$A_L$ is an unbounded Fredholm
self-adjoint operator on $L^2(X,E)$ with discrete spectrum and with
domain consisting of sections $u$ belonging to the Sobolev space
$H^1(X,E)$ and satisfying condition \eqref{e3} \textup(in which $u
|_Y$ is treated as the element of $H^{1/2}(Y,E_Y)$ obtained from
$u$ by restriction to~$Y$ by virtue of the trace theorem and
$\pi_L$ is treated as a mapping $\pi_L\colon
H^{1/2}(Y,E_Y)\longrightarrow H^{1/2}(Y,E_Y\slash L)$\textup).

Now assume that both the Dirac type operator $A$ \eqref{e1} and the
subbundle $L$ continuously depend on a parameter $t\in[0,1]$
(namely, the coefficients of $A$ and $\pi_L$ depend on $t$
continuously together with all of their
derivatives\footnote{Apparently, one derivative would suffice, but
let us think big.}); i.e., $A=A(t)$ and $L=L(t)$. Moreover, assume
that condition \eqref{e2} holds for each $t$. Then, by Theorem~7.16
in~\cite{BLZ1}, the operator $A(t)_{L(t)}$ continuously depends on
$t$ in the topology of uniform resolvent convergence, and
Definition~\ref{def-sf} specifies the spectral flow
$\operatorname{sf}\{A(t)_{L(t)}\}$ of the family $\{A(t)_{L(t)}\}$,
$t\in[0,1]$.

\medskip

Next, let an automorphism $U\colon E\to E$ of the bundle $E$ be
given such that
\begin{equation}\label{e5}
    A(1)=U A(0) U^{-1},\qquad U(L(0))=L(1).
\end{equation}
Then $U A(0)_{L(0)}U^{-1}=A(1)_{L(1)}$; i.e., the operators
$A(0)_{L(0)}$ and $A(1)_{L(1)}$ are similar and hence isospectral,
so that the spectral flow of the family $\{A(t)_{L(t)}\}$ is a
homotopy invariant (in the class of families satisfying a condition
of the form~\eqref{e5}). Thus, it is natural to pose the problem of
computing it in topological terms.

\medskip

To do this, we introduce an auxiliary elliptic boundary value
problem on the Cartesian product $X\times S^1$ of the manifold $X$
by the circle $S^1$ (see Fig.~\ref{fig03}). Namely, let us define a
bundle $\mathcal{E}$ over $X\times S^1$ as follows. Take the
pullback of $E$ to the product $X\times [0,1]$ via the natural
projection $\pi\colon X\times [0,1]\to X$ and then use the
automorphism $U\colon (\pi^*E)_{X\times\{0\}}\longrightarrow
(\pi^*E)_{X\times\{1\}}$ as the clutching automorphism.\footnote{We
assume the circle $S^1$ to be obtained from the interval $[0,1]$ by
gluing together the endpoints.}  By conditions \eqref{e5}, the
family $\{A(t)\}$ specifies a well-defined differential operator on
the space of sections of the bundle $\mathcal{E}$, while the family
of subbundles $L(t)$ defines a subbundle $\mathcal{L}\subset
\mathcal{E}_{Y\times S^1}$ in the restriction of $\mathcal{E}$ to
the boundary $Y\times S^1$ of the manifold $X\times S^1$.
\begin{figure}
\centering
  \includegraphics{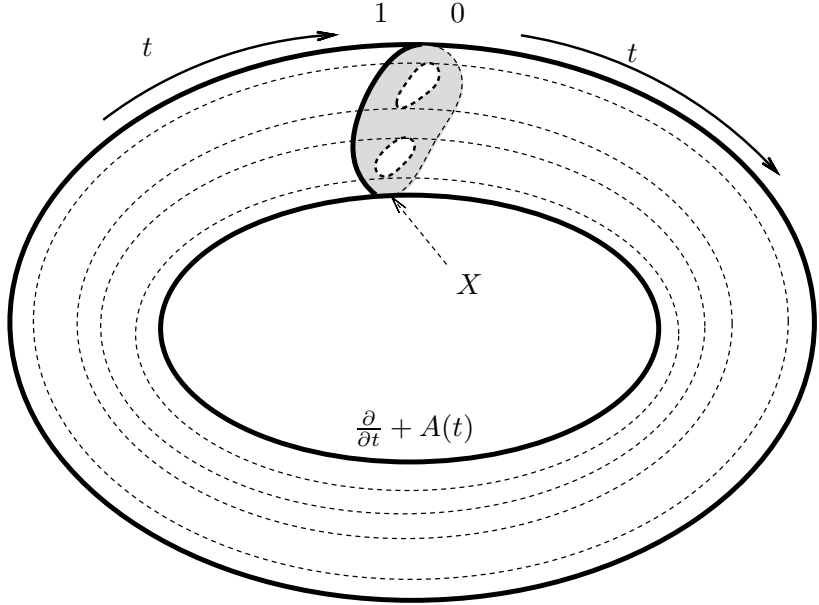}
 \vglue5pt
\caption{The manifold $X\times S_1$ obtained by gluing together the
faces $t=0$ and $t=1$ of the product $X\times [0,1]$, and the operator
$\frac{\partial}{\partial t}+A(t)$.\label{fig03}}
\end{figure}

\begin{proposition}\label{p1}
The operator
\begin{equation}\label{e6}
     \frac{\partial }{\partial t}+A(t)\; \colon C^\infty(X\times S^1,\mathcal{E})
     \longrightarrow C^\infty(X\times S^1,\mathcal{E})
\end{equation}
is elliptic, and the boundary conditions
\begin{equation}\label{e7}
     \pi_\mathcal{L}(u |_{Y\times S^1})=0,\quad\text{where }\pi_\mathcal{L}\colon
     \mathcal{E}_{Y\times S^1}\longrightarrow \mathcal{E}_{Y\times S^1}\slash
     \mathcal{L}\quad
     \text{is the natural projection,}
\end{equation}
defined by the subbundle $\mathcal{L}$, are elliptic for the
operator \eqref{e6}. The closure $\bigl(\frac{\partial
}{\partial t}+A(t)\bigr)_\mathcal{L}$ of the operator \eqref{e6}
from the domain specified by conditions \eqref{e7} is an unbounded
Fredholm operator on $L^2(X\times S^1,\mathcal{E})$ with domain
$\mathfrak{D}_\mathcal{L}$ consisting of the sections $u\in
H^1(X\times S^1,\mathcal{E})$ satisfying condition \eqref{e7}.
\end{proposition}

Now we are in a position to state the main theorem of the present
paper.

\begin{theorem}\label{t1}
One has
\begin{equation}\label{e8}
    \operatorname{sf}\{A(t)_{L(t)}\}
    =\operatorname{ind} \biggl(\frac{\partial }{\partial t}
    +A(t)\biggr)_\mathcal{L}.
\end{equation}
\end{theorem}
The right-hand side of \eqref{e8} is the \textit{analytic index} of
the operator $\bigl(\frac{\partial }{\partial
t}+A(t)\bigr)_\mathcal{L}$, i.e., the difference of dimensions of
its kernel and cokernel, which can be expressed in topological
terms by the Atiyah--Bott formula \cite{AtBo2} (see
also~\cite[Sec.~20.3]{Hor3}).

\section{Proof of the Main Assertions}

\begin{proof}[Proof of Proposition~\textup{\ref{p1}}]
Consider the operator
\begin{equation}\label{e10}
    \mathfrak{A}=\begin{pmatrix}
      0 & \frac{\partial }{\partial t}+A(t) \\
      -\frac{\partial }{\partial t}+A(t) & 0 \\
    \end{pmatrix}\colon C^\infty(X\times S^1,\mathcal{E}\oplus\mathcal{E})
    \longrightarrow C^\infty(X\times S^1,\mathcal{E}\oplus\mathcal{E}).
\end{equation}
This is a total formally self-adjoint Dirac type operator on
$X\times S^1$ with symbol
\begin{equation}\label{e11}
    \sigma_\mathfrak{A}(x,t,\xi,\xi_0)=\begin{pmatrix}
      0 & i\xi_0I+\sigma_{A(t)}(x,\xi) \\
      -i\xi_0I+\sigma_{A(t)}(x,\xi) & 0 \\
    \end{pmatrix},
\end{equation}
where $\xi_0$ is the momentum variable conjugate to $t\in S^1$, and
the operator $\frac{\partial}{\partial t}+A(t)$ is its chiral part.
The subbundle
$\mathfrak{L}=\mathcal{L}\oplus\mathcal{L}\subset(\mathcal{E}\oplus\mathcal{E})_{Y\times
S^1}$ satisfies a condition of the form \eqref{e2} with respect to
$\sigma_\mathfrak{A}$ and hence specifies self-adjoint elliptic
boundary conditions for $\mathfrak{A}$. Indeed, the conormal vector
to the boundary of $X\times S^1$ at an arbitrary point $(x,t)\in
Y\times S^1$ has the form $\mathfrak{n}(x,t)=(0,\mathbf{n}(x))$,
where $\mathbf{n}(x)$ is the conormal vector to the boundary of $X$
itself and $\mathcal{L}_{(x,t)}=L(t)_x$; hence, for any
\begin{equation*}
    v={}^t(v_1,v_2)\in\mathfrak{L}_{(x,t)},\quad w={}^t(w_1,w_2)
    \in\mathfrak{L}_{(x,t)},\qquad\text{i.e., $v_1,v_2,w_1,w_2\in
    L(t)_x$},
\end{equation*}
we have
\begin{equation*}
    (v,\sigma_\mathfrak{A}(x,t,\mathfrak{n}(x,t))w)=
      (v_1,\sigma_{A(t)}(x,\mathbf{n}(x))w_2)+(v_2,\sigma_{A(t)}(x,\mathbf{n}(x))w_1)=0,
\end{equation*}
because condition \eqref{e2} is satisfied for $A(t)$ and the bundle
$L(t)$. This, again by virtue of the results in~\cite{BrLe3}
and~\cite[Chaps.~18 and~19]{BBW1}, implies the claim of
Proposition~\ref{p1} first for the operator~$\mathfrak{A}$ and
then, as a consequence, for its chiral part~$\frac{\partial
}{\partial t}+A(t)$.
\end{proof}

\begin{proof}[Proof of Theorem~\ref{t1}]
\textbf{a.} Without loss of generality, we assume that
$0\notin\operatorname{Spec}(A(0)_{L(0)})$. (Otherwise, one can
replace the operator $A(t)$ by $A(t)+\varepsilon$ with small real
$\varepsilon$, which changes neither the left- nor the right-hand
side of~\eqref{e8}.)

\medskip

\textbf{b.} Also without loss of generality, we assume throughout
the following that the subbundle $L(t)$ is independent of the
parameter $t$, $L(t)=L(0)\equiv L$, $t\in[0,1]$. Indeed, let
$V(t)\colon E_Y\to E_Y$ be a family of unitary automorphisms of
$E_Y$ such that $V(0)=I$ and $L(t)=V(t)L$, $t\in[0,1]$. (One can
readily construct such a family by solving the Cauchy problem $\dot
V=[P,\dot P]V$, $V(0)=I$, where $P=P(t)$ is the projection onto
$L(t)$ in $E_Y$.) This family can be continued (by a homotopy to
the identity mapping along the variable normal to the boundary) to
a family of unitary automorphisms $W(t)\colon E\to E$ such that
$W(t)\big|_Y=V(t)$. Set
\begin{equation*}
    \widetilde A(t)=W^{-1}(t)A(t)W(t);
\end{equation*}
then, obviously,
\begin{gather*}
    W^{-1}(1)U\widetilde A(0)(W^{-1}(1)U)^{-1}=W^{-1}(1)U\widetilde
    A(0)U^{-1}W(1)=W^{-1}(1)A(1)W(1)=\widetilde A(1),\\
    W^{-1}(1)UL\equiv W^{-1}(1)UL(0)=W^{-1}(1)L(1)=L(0)\equiv L;
\end{gather*}
i.e., conditions of the form \eqref{e5} are satisfied for the
family $\{\widetilde A(t)\}$ and the constant family of subspaces
$\widetilde L(t)=L$ if one takes the automorphism $\widetilde
U=W^{-1}(1)U$. Furthermore,
\begin{equation}\label{e104}
    \operatorname{sf}\{A(t)_{L(t)}\}=\operatorname{sf}\{\widetilde
    A(t)_L\},
\end{equation}
because the operators $A(t)$ and $\widetilde A(t)$ are similar.
Next, the family $W(t)$ generates a bundle isomorphism
$\mathcal{W}\colon\widetilde{\mathcal{E}}\to\mathcal{E}$, where the
bundle $\widetilde{\mathcal{E}}$ over $X\times S^1$, by analogy
with $\mathcal{E}$, is obtained from the pullback of $E$ to
$X\times [0,1]$ by clutching with automorphism $\widetilde U$. The
operator
\begin{equation}\label{e101}
    \mathcal{W}^{-1}\biggl(\frac{\partial }{\partial t}
            +A(t)\biggr)_\mathcal{L}\mathcal{W}
    =\biggl(\frac{\partial }{\partial t}+\widetilde A(t)-W^{-1}(t)
    \frac{\partial W}{\partial t}(t)\biggr)_{\widetilde{\mathcal{L}}}
\end{equation}
has the same index as $\bigl(\frac{\partial }{\partial
t}+A(t)\bigr)_\mathcal{L}$ and acts on the space of sections of
$\widetilde{\mathcal{E}}$ satisfying the boundary condition
associated with the subbundle
$\widetilde{\mathcal{L}}=\mathcal{W}^{-1}\big|_{Y\times
S^1}\mathcal{L}$, for which
$\widetilde{\mathcal{L}}_t=V(t)^{-1}L(t)=L$ for all $t\in S^1$.
Finally, the homotopy
\begin{equation}
     \biggl(\frac{\partial }{\partial t}+\widetilde A(t)
     -\lambda W^{-1}(t)\frac{\partial W}{\partial t}(t)\biggr)
    _{\widetilde{\mathcal{L}}},\qquad \lambda\in[0,1],
\end{equation}
in the class of Fredholm operators reduces the operator
\eqref{e101} for $\lambda=0$ to $\bigl(\frac{\partial }{\partial
t}+\widetilde A(t)\bigr)_{\widetilde{\mathcal{L}}}$, so that
\begin{equation*}
    \operatorname{ind} \biggl(\frac{\partial }{\partial t}+A(t)\biggr)_\mathcal{L}
    =\operatorname{ind}\biggl(\frac{\partial }{\partial t}+\widetilde
    A(t)\biggr)_{\widetilde{\mathcal{L}}},
\end{equation*}
which, together with \eqref{e104}, completes reduction to the case
of a bundle $L(t)=L$ independent of $t$. We omit the tilde over
letters in what follows.

\textbf{c.} In the proof, we need a family of operators on the
infinite cylinder $X\times\mathbb{R}$. Let us describe it. The
pullbacks of the bundle $E$ from $X$ to $X\times\mathbb{R}$ and the
bundle $L$ from $Y$ to $Y\times\mathbb{R}$ will be denoted by the
same letters $E$ and $L$, respectively; this shall not lead to
confusion. The coordinate on the line $\mathbb{R}$ will be denoted
by $t$. For $\alpha,\beta\in\mathbb{R}$, we introduce the weighted
spaces $L^2_{\alpha\beta}(X\times\mathbb{R},E)$ and
$H^1_{\alpha\beta}(X\times\mathbb{R},E)$ of sections $u$ of $E$
with finite norm
\begin{align*}
    \left\| u \right\|_{0,\alpha\beta}&=\biggl\{\int_{-\infty}^0
    \left\|{u(t)} \right\|_{L^2(X,E)}^2e^{2\alpha t}\,dt+\int_0^{\infty}
    \left\|{u(t)} \right\|_{L^2(X,E)}^2e^{2\beta t}\,dt\biggr\}^{1/2}\qquad
    \text{and}\\
    \left\|{u} \right\|_{1,\alpha\beta}&=\biggl\{\int_{-\infty}^0
    \biggl(\left\|{\vphantom{\Big|}\smash{\frac{\partial u(t)}{\partial t}}} \right\|_{L^2(X,E)}^2
    +\left\|{u(t)} \right\|_{H^1(X,E)}^2\biggr)e^{2\alpha t}\,dt\\&\qquad+
    \int_0^{\infty}
    \biggl(\left\|{\vphantom{\Big|}\smash{\frac{\partial u(t)}{\partial t}}} \right\|_{L^2(X,E)}^2
    +\left\|{u(t)} \right\|_{H^1(X,E)}^2\biggr)e^{2\beta
    t}\,dt\biggr\}^{1/2},
\end{align*}
respectively. In particular,
$H^1_{\alpha\beta}(X\times\mathbb{R},E)\subset
L^2_{\alpha\beta}(X\times\mathbb{R},E)$. By
$\mathfrak{D}_{\alpha\beta}$ we denoted the closed subspace of
$H^1_{\alpha\beta}(X\times\mathbb{R},E)$ consisting of the sections
satisfying the boundary conditions determined by $L$; i.e.,
\begin{equation}\label{e12}
    \mathfrak{D}_{\alpha\beta}=\bigl\{u\in H^1_{\alpha\beta}(X\times\mathbb{R},E)\colon
    \pi_Lu=0\bigr\}.
\end{equation}
Let $0\le \theta\le1$. Set\footnote{%
Here, just as above and below, we for brevity omit the standard
smoothing procedure eliminating the jumps of the derivatives (in
the present case, for $t=0$ and $t=\theta$) when describing the
homotopies.}
\begin{equation*}
    \tau(t,\theta)=\begin{cases}
    0,&t\le0,\\
    t,&0\le t\le\theta,\\
    \theta,&\theta\le t.\\
    \end{cases}
\end{equation*}
For $\gamma\in\mathbb{R}$, let
\begin{equation}\label{e15}
    \mathcal{A}(\theta,\gamma)=\frac{\partial }{\partial t}+A(\tau(t,\theta)) \colon
    L^2_{0\gamma}(X\times\mathbb{R},E)\longrightarrow L^2_{0\gamma}(X\times\mathbb{R},E)
\end{equation}
be the operator with domain $\mathfrak{D}_{0\gamma}$ (see
Fig.~\ref{fig04}). Let us state a number of properties of the
operators $\mathcal{A}(\theta,\gamma)$.
\begin{figure}
\centering
  \includegraphics{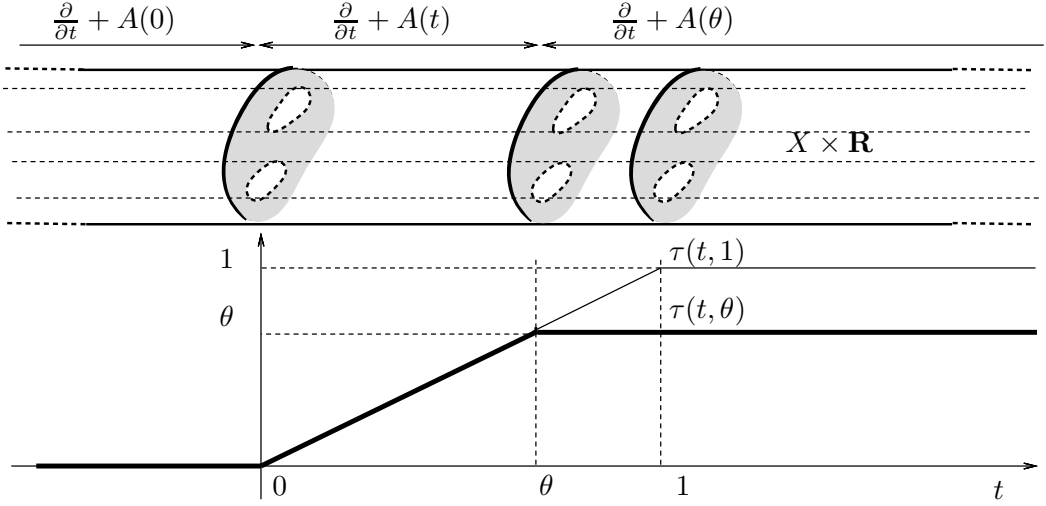}
 \vglue5pt
\caption{The operator $\mathcal{A}(\theta,\gamma)=\frac{\partial
}{\partial t}+A(\tau(t,\theta))$ on the infinite cylinder
$X\times\mathbf{R}$.\label{fig04}}
\end{figure}
\begin{lemma}\label{l1}
The operator $\mathcal{A}(\theta,\gamma)$ is Fredholm for $\theta$
such that $\gamma\notin\operatorname{Spec}(A(\theta)_{L})$, and
$\operatorname{ind}\mathcal{A}(\theta,\gamma)$ is a locally
constant function of $\theta$ on the set of such values
of~$\theta$.
\end{lemma}
\begin{lemma}\label{l2}
If
$\gamma,\widetilde\gamma\notin\operatorname{Spec}(A(\theta)_{L})$
and $\gamma>\widetilde\gamma$, then the difference
$\operatorname{ind}\mathcal{A}(\theta,\widetilde\gamma)-\operatorname{ind}\mathcal{A}(\theta,\gamma)$
is equal to the number of eigenvalues \textup(counting
multiplicities\textup) of the operator $A(\theta)_{L}$ on the
interval $(\widetilde\gamma,\gamma)$.
\end{lemma}
\begin{lemma}\label{l3}
$\operatorname{ind}\mathcal{A}(0,0)=0$.
\end{lemma}
\begin{lemma}\label{l4}
$\operatorname{ind}\mathcal{A}(1,0)=\operatorname{ind}\bigl(\frac{\partial
}{\partial t}+A(t)\bigr)_\mathcal{L}$.
\end{lemma}

The proof of Lemmas~\ref{l1}--\ref{l4} will be given below. Now let
us show that these lemmas imply the claim of the theorem. The
spectral flow of the family $\{A(t)_{L}\}$ is given by
Definition~\ref{def-sf} for some partition
$0=t_0<t_1<\dotsm<t_{n+1}=1$ and real numbers
$\gamma_1,\dots,\gamma_{n+1}=\gamma_1$, and we can assume that
$\gamma_1=0$ (because we have assumed that
$0\notin\operatorname{Spec}(A(0)_{L})$). Let $m_j$ be the number of
eigenvalues (counting multiplicities) of the operator $A(t_j)_{L}$
in the interval between $\gamma_j$ and $\gamma_{j+1}$. It follows
from Lemma~\ref{l1} that the operator
$\mathcal{A}(\theta,\gamma_j)$ is Fredholm for
$\theta\in[t_{j-1},t_j]$ and
\begin{equation}\label{e18}
    \operatorname{ind}\mathcal{A}(t_j,\gamma_j)
    -\operatorname{ind}\mathcal{A}(t_{j-1},\gamma_j)=0,
    \qquad j=1,\dotsc,n+1.
\end{equation}
By Lemma~\ref{l2},
\begin{equation}\label{e19}
    \operatorname{ind}\mathcal{A}(t_j,\gamma_{j+1})-\operatorname{ind}
    \mathcal{A}(t_j,\gamma_j)=m_j\operatorname{sign}(\gamma_j-\gamma_{j+1}),
    \qquad j=1,\dotsc,n.
\end{equation}
By summing relations \eqref{e18} and \eqref{e19} over all
corresponding $j$, by adding the results, and by taking into
account Lemma \ref{l3} and the relation $\gamma_1=\gamma_n=0$, we
obtain
\begin{equation}\label{e20}
    \operatorname{ind}\mathcal{A}(1,0)=\operatorname{ind}\mathcal{A}(1,0)
    -\operatorname{ind}\mathcal{A}(0,0)
    =\sum_{j=1}^nm_j\operatorname{sign}(\gamma_j-\gamma_{j+1})=
    \operatorname{sf}\{A(t)_{L}\}.
\end{equation}
It remains to use Lemma~\ref{l4}. The proof of Theorem~\ref{t1} is
complete.
\end{proof}

Now let us prove Lemmas \ref{l1}--\ref{l4}.
\begin{proof}[Proof of Lemma~\textup{\ref{l1}}.]
To prove that the operator $\mathcal{A}(\theta,\gamma)$ is
Fredholm, it suffices to construct a regularizer, i.e., an operator
\begin{equation*}
    \mathcal{R}\colon L^2_{0\gamma}(X\times\mathbb{R},E)
                  \longrightarrow \mathfrak{D}_{0\gamma}
\end{equation*}
such that the operators $I-\mathcal{A}(\theta,\gamma)\mathcal{R}$
and $I-\mathcal{R}\mathcal{A}(\theta,\gamma)$ are compact in the
spaces $L^2_{0\gamma}(X\times\mathbb{R},E)$ and
$\mathfrak{D}_{0\gamma}$, respectively. This can be done by the
frozen-coefficients technique, standard in elliptic theory. (In our
case, we ``freeze'' the variable $t$). To this end, for given
$\tau\in[0,1]$ and $\nu\in\mathbb{R}$, consider the operator
\begin{equation}\label{e22}
    \frac{\partial }{\partial t}+A(\tau) \colon L^2_{\nu\nu}(X\times\mathbb{R},E)
    \longrightarrow L^2_{\nu\nu}(X\times\mathbb{R},E)
    \quad\text{with domain $\mathfrak{D}_{\nu\nu}$.}
\end{equation}
The operator \eqref{e22} is invertible provided that
$\nu\notin\operatorname{Spec}(A(\tau)_{L})$. The inverse operator
$R_{\nu}(\tau)$ is given by the formula
\begin{equation}\label{e24}
    [R_{\nu}(\tau)u](t)=\frac1{\sqrt{2\pi}}\int_{\operatorname{Im} p=\nu}
    e^{ipt}(ip+A(\tau)_{L})^{-1}\widetilde u(p)\,dp,\qquad
    u\in L^2_{\nu\nu}(X\times\mathbb{R},E),
\end{equation}
where $\widetilde u(p)$, $\operatorname{Im} p=\nu$, is the Fourier
transform of $u$ with respect to the variable $t$. Consider a
finite cover $\{U_j\}_{j=1}^s$ of $[0,1]$ by open intervals such
that $\nu_j\notin\operatorname{Spec}(A(\tau(t,\theta))_L)$ for
$t\in U_j$ for some real numbers $\nu_j$; let $U_0=(-\infty,0)$,
$\nu_0=0$, $U_{s+1}=(1,\infty)$, and $\nu_{s+1}=\gamma$, and let
$1=\sum_{j=0}^{s+1}\psi_j^2$ be a smooth partition of unity
subordinate to the cover of the line $\mathbb{R}$ by the sets
$U_j$, $j=0,\dotsc,s+1$. Then $\mathcal{R}$ can be defined by the formula
\begin{equation*}
    [\mathcal{R} u](t)=\sum_{j=0}^{s+1}\psi_j(t)
    \bigl[[\mathcal{R}_{\nu_j}(\tau)(\psi_ju)](t)\bigr]\big|_{\tau=\tau(t,\theta)}.
\end{equation*}
(Note that $\mathcal{R}$ is well defined as an operator from
$L^2_{0\gamma}(X\times\mathbb{R},E)$ to $\mathfrak{D}_{0\gamma}$,
because the operator of multiplication by $\psi_j$ is continuous
from $L^2_{0\gamma}(X\times\mathbb{R},E)$ to
$L^2_{\nu_j\nu_j}(X\times\mathbb{R},E)$ and from
$\mathfrak{D}_{\nu_j\nu_j}$ to $\mathfrak{D}_{0\gamma}$. For
$j\ne0,s+1$, this follows from the compactness of the support of
$\psi_j$; for $j=0$, from the fact that $\nu_0=0$ and $\psi_0(t)=0$
for $t>0$; for $j=s+1$, from the fact that $\nu_{s+1}=\gamma$ and
$\psi_{s+1}(t)=0$ for $t<0$.) Now a straightforward computation
shows that
\begin{equation*}
    \mathcal{A}(\theta,\gamma)\mathcal{R} u(t)=u(t)
    +\sum_{j=0}^{s+1}\biggl[\biggl[\biggl(\frac{\partial \psi_j}{\partial t} (t)
    \mathcal{R}_{\nu_j}(\tau) +\psi_j(t)\frac{\partial \tau}{\partial t}
    (t,\theta)\frac{\partial \mathcal{R}_{\nu_j}}{\partial \tau}(\tau)\biggr)(\psi_ju)\biggr]
               (t)\biggr]\bigg|_{\tau=\tau(t,\theta)}.
\end{equation*}
Since the functions $\frac{\partial \psi_j}{\partial t}(t)$ and
$\frac{\partial \tau}{\partial t} (t,\theta)$ are compactly
supported, it follows from standard facts about embeddings of
Sobolev spaces that the second term on the right-hand side defines
a compact operator on $L^2_{0\gamma}(X\times\mathbb{R},E)$. In a
similar way, one can study the product
$\mathcal{R}\mathcal{A}(\theta,\gamma)$.

The local constancy of the index of $\mathcal{A}(\theta,\gamma)$ as
a function of $\theta$ follows from the fact that this operator
continuously depends on $\theta$ in the operator norm as an
operator from $\mathfrak{D}_{0\gamma}$ to
$L^2_{0\gamma}(X\times\mathbb{R},E)$. The proof of Lemma \ref{l1}
is complete.
\end{proof}

\begin{proof}[Proof of Lemma \textup{\ref{l3}}]
This lemma is the special case of the invertibility of the operator
\eqref{e22} for $\nu=0$ and $\tau=0$.
\end{proof}

The proof of Lemmas~\ref{l2} and~\ref{l4} is based on the
localization method (the index locality principle; see
\cite[Theorem~4.10]{NSScS99} and also \cite{AAA,R:NaSt3,R:NaSt7}
and the references therein). Having in mind our goals, let us state
the claim of Theorem~4.10 in \cite{NSScS99} for the simplest
special case.

\medskip

\textit{Let $N_1,N_2\subset M$ be disjoint closed subsets of a
manifold $M$, and let $D\colon
\mathcal{H}_1\longrightarrow\mathcal{H}_2$ be a bounded Fredholm
operator acting on some Hilbert spaces of sections of bundles over
$M$. Next, let $\varkappa\colon M\to[0,1]$ be a smooth mapping such
that $N_1\subset f^{-1}(0)$ and $N_2\subset f^{-1}(1)$, and let
$\mathcal{C}\subset C^\infty(M)$ be a subalgebra consisting of
functions constant on $N_1$ and on $N_2$ and containing all
functions of the form $\psi(x)=\varphi(\varkappa(x))$, where
$\varphi$ is a smooth function on $[0,1]$. Suppose that the
commutator of $D$ with the operator of multiplication by any
function in $\mathcal{C}$ is compact. Then the index increments
arising from changes of $D$ on $N_1$ and $N_2$ preserving the
Fredholm property and the compactness of commutators\footnote{The
changes may affect not only the operator itself but also the spaces
on which it acts and even the very manifold (e.g., cutting away
some parts and pasting another ones); all these changes should
occur strictly inside the corresponding set $N_j$, and everything
on $M\setminus N_j$ should remain unchanged.} are
independent\textup:
\begin{equation*}
  \Delta_{N_1\sqcup N_2}=\Delta_{N_1}+\Delta_{N_2},
\end{equation*}
where
\begin{itemize}\itemsep0pt
    \item $\Delta_{N_1}$ is the index increment occurring if the operator
    is changed \textit{only} on $N_1$.
    \item $\Delta_{N_2}$ is the index increment occurring if the operator
    is changed \textit{only} on $N_2$.
    \item $\Delta_{N_1\sqcup N_2}$ is the index increment
    occurring if the operator
    is simultaneously changed
    \textit{both on $N_1$ and $N_2$}.
\end{itemize}}

\medskip

The practical application of the localization method to the proof
of Lemmas~\ref{l2} and~\ref{l4} implements the following idea. We
wish to compute how the index of some operator $D$ changes for a
given change of the operator on a set $N_1$, but it is difficult to
compute the index increment owing to the complicated structure of
$D$ outside $N_1$. Let us modify the operator $D$ on some set $N_2$
disjoint with $N_1$ so as to obtain an operator $\widetilde D$ of
simpler structure whose index increment under the given change on
$N_1$ can be computed. This increment coincides with the desired
increment for the original operator.

\begin{proof}[Proof of Lemma~\textup{\ref{l2}}]
Take $X\times\mathbb{R}$ for the manifold $M$, the set $\{t\ge 2\}$
for $N_1\subset M$, the set $\{t\le 1\}$ for $N_2\subset M$, and
the algebra of infinitely differentiable functions $\varphi(t)$ of
$t\in\mathbb{R}$ constant on $N_1$ and on $N_2$ for the function
algebra $\mathcal{C}$. The original operator $D$ is the operator
$A(\theta,\gamma)$, which we treat as a bounded Fredholm operator
on the spaces
\begin{equation}\label{e110}
    D=A(\theta,\gamma)\colon \mathfrak{D}_{0\gamma}
    \longrightarrow L_{0\gamma}^2(X\times \mathbb{R},E),
\end{equation}
and we need to compute the index increment for this operator if
$\gamma$ is replaced by $\widetilde\gamma$. Note that the
commutator of the operator \eqref{e110} with a smooth function
$\varphi\in\mathcal{C}$ is the operator of multiplication by the
compactly supported function $\varphi'(t)$, which is compact as an
operator from $\mathfrak{D}_{0\gamma}$ to $L_{0\gamma}^2(X\times
\mathbb{R}^1,E)$, so that we are just in a position to use the
localization method. The replacement of $\gamma$ by
$\widetilde\gamma$ changes the operator $D$ only on the set $N_1$.
(The expression specifying the operator and the boundary conditions
remain the same, but the spaces where the operator acts are
changed, the change being solely concerned with the admissible
growth of functions as $t\to+\infty$; i.e., in particular, the
restriction of these spaces to $M\setminus N_1$ is unchanged at
all.) Now let us replace $D$ by the operator
\begin{equation}\label{e111}
    \widetilde D=\frac{\partial }{\partial t}+A(\theta)
    \colon\mathfrak{D}_{\gamma\gamma}
    \longrightarrow L_{0\gamma}^2(X\times\mathbb{R},E).
\end{equation}
This operator differs from $D$ (both in the differential expression
and in the spaces where it acts) only on $N_2$. Thus, it suffices
to compute the index increment for this operator under the change
on $N_1$ the same as for $D$. The operator $\widetilde D$ is
invertible, so that $\operatorname{ind}\widetilde D=0$. The change
of this operator on $N_1$ results in the operator
\begin{equation}\label{e102}
    \frac{\partial }{\partial t}+A(\theta)
    \colon \mathfrak{D}_{\gamma\widetilde\gamma}\longrightarrow
    L^2_{\gamma\widetilde\gamma}(X\times\mathbb{R},E);
\end{equation}
thus, it remains to compute the index of the latter. For this
computation, it is convenient to treat the operator \eqref{e102} as
an unbounded Fredholm operator on
$L^2_{\gamma\widetilde\gamma}(X\times\mathbb{R},E)$ with domain
$\mathfrak{D}_{\gamma\widetilde\gamma}$. Then the adjoint operator
has the form $-\frac{\partial }{\partial t}+A(\theta)$ and acts on
the dual space
$L^2_{-\gamma,-\widetilde\gamma}(X\times\mathbb{R},E)$ with domain
$\mathfrak{D}_{-\gamma,-\widetilde\gamma}$. The elements of the
null space of the operator \eqref{e102} should have the form
$v(x)e^{-\lambda t}$, where $\lambda$ is an eigenvalue of
$A(\theta)_L$ and $v(x)$ is one of the corresponding
eigenfunctions. The condition that these elements belong to the
weighted space $L^2_{\gamma\widetilde\gamma}(X\times\mathbb{R},E)$
implies that $\widetilde\gamma<\lambda<\gamma$. Thus, the dimension
of the null space is equal to the number of eigenvalues
\textup(with regard of multiplicity\textup) of the operator
$A(\theta)_{L}$ in the interval $(\widetilde\gamma,\gamma)$. The
elements of the null space of the adjoint operator should have the
form $v(x)e^{\lambda t}$ and belong to the weighted space
$L^2_{-\gamma,-\widetilde\gamma}(X\times\mathbb{R},E)$). It follows
that $\gamma<\lambda<\widetilde\gamma$, but this is impossible,
because $\gamma>\widetilde\gamma$. Thus, the null space of the
adjoint operator is trivial, the index of the operator~\eqref{e102}
coincides with the dimension of its null space, and we arrive at
the assertion of Lemma~\ref{l2}.
\end{proof}

\begin{proof}[Proof of Lemma~\ref{l4}]
We need to prove that the operators
\begin{equation}\label{e120}
    \mathcal{A}(1,0)\colon \mathfrak{D}_{00}\longrightarrow L^2(X\times\mathbb{R},E)
\end{equation}
and
\begin{equation}\label{e121}
    \biggl(\frac{\partial }{\partial t}+A(t)\biggr)_\mathcal{L}\colon
    \mathfrak{D}_\mathcal{L}\longrightarrow L^2(X\times
    S^1,\mathcal{E})
\end{equation}
have the same index. We assume (this can always be achieved by a
homotopy) that $A(t)=A(1)$ for $t\ge1-2\varepsilon$ and $A(t)=A(0)$
for $t\le2\varepsilon$ for some given $\varepsilon>0$. Set
\begin{equation*}
    N_1=\{t\in (-\infty,\varepsilon]\cup[1-\varepsilon,\infty)\},\qquad
    N_2=\{t\in[2\varepsilon,1-2\varepsilon]\}.
\end{equation*}
For the function algebra $\mathcal{C}$ we again take the algebra of
infinitely differentiable functions $\varphi(t)$ of $t$ constant on
$N_1$ and $N_2$. The operator \eqref{e121} can be obtained from the
operator \eqref{e120} by the following change on $N_1$: one cuts
away and disposes of the half-cylinders $X\times(-\infty,0)$ and
$X\times(1,\infty)$, and the faces $X\times\{0\}$ and
$X\times\{1\}$ of the remaining product $X\times[0,1]$ are glued
together, the bundle $E$ giving rise to the bundle $\mathcal{E}$
via the clutching automorphism $U\colon E|_{t=0}\longrightarrow
E|_{t=1}$.

We should show that the index increment for this change of the
operator on $N_1$ is zero. To this end, we replace the operator
\eqref{e120} by the operator
\begin{equation}\label{e122}
 \mathcal{A}(0,0)\colon \mathfrak{D}_{00}\longrightarrow L^2(X\times\mathbb{R},E).
\end{equation}
The operator \eqref{e122} will differ from the operator
\eqref{e120} only on $N_2$ if we rewrite the former in the
equivalent form
\begin{equation}\label{e123}
    \mathcal{A}(0,0)=\begin{cases}
    \frac{\partial }{\partial t}+A(0) &\text{for $t\le\frac12$,}\\
    \frac{\partial }{\partial t}+A(1) &\text{for $t\ge\frac12$,}
    \end{cases}
\end{equation}
where it is assumed that the bundle in which the operator
\eqref{e123} acts is obtained from $E|_{t\le\frac12}$ and
$E|_{t\ge\frac12}$ by the standard clutching construction at
$t=\frac12$ with the automorphism
\begin{equation*}
    U^{-1}\colon E|_{t=1/2+0}\longrightarrow E|_{t=1/2-0}.
\end{equation*}
(The passage from \eqref{e122} to \eqref{e123} is essentially none
other than rewriting the operator $\mathcal{A}(0,0)$ for
$t\ge\frac12$ in ``new coordinates'' in the fibers of $E$.)

Now let us change the operator \eqref{e122} written in the form
\eqref{e123} on $N_1$ in the same way as we have earlier changed
the operator \eqref{e120}. The resulting operator on $X\times S_1$
has the form \eqref{e123}, and the bundle in which it acts is
obtained from $E$ by clutching construction with the automorphism
\begin{equation*}
    U\colon E|_{t=0}\longrightarrow E|_{t=1},\qquad
    U^{-1}\colon E|_{t=1/2+0}\longrightarrow E|_{t=1/2-0}.
\end{equation*}
It is easily seen that this bundle is isomorphic to the pullback of
$E$ on $X\times S_1$, and the resulting operator itself is none
other than $\bigl(\frac{\partial }{\partial t}+A(0)\bigr)_L$; its
index is zero, because it is invariant with respect to rotations
along $S^1$. The index of the operator~\eqref{e120} is zero as well
(it is invertible), so that the index increment is zero, and the
proof of the lemma is complete.
\end{proof}

\begin{proof}[Proof of Theorem~\ref{co1}]
Theorem~\ref{t1} shows that the spectral flow of the family
\eqref{ei1a} obeys the localization principle: any modifications
applied to the Dirac operator \eqref{ei1} in the planar domain $X$
and to the function $B$ occurring in the boundary conditions
automatically lift to $X\times S_1$ becoming modifications of
$\frac{\partial
}{\partial t}+D_t$; the latter enjoy the index locality principle
\cite[Theorem~4.10]{NSScS99}, and the index is equal to the
spectral flow by Theorem~\ref{t1}.

\begin{figure}
\centering
 \includegraphics[scale=0.6]{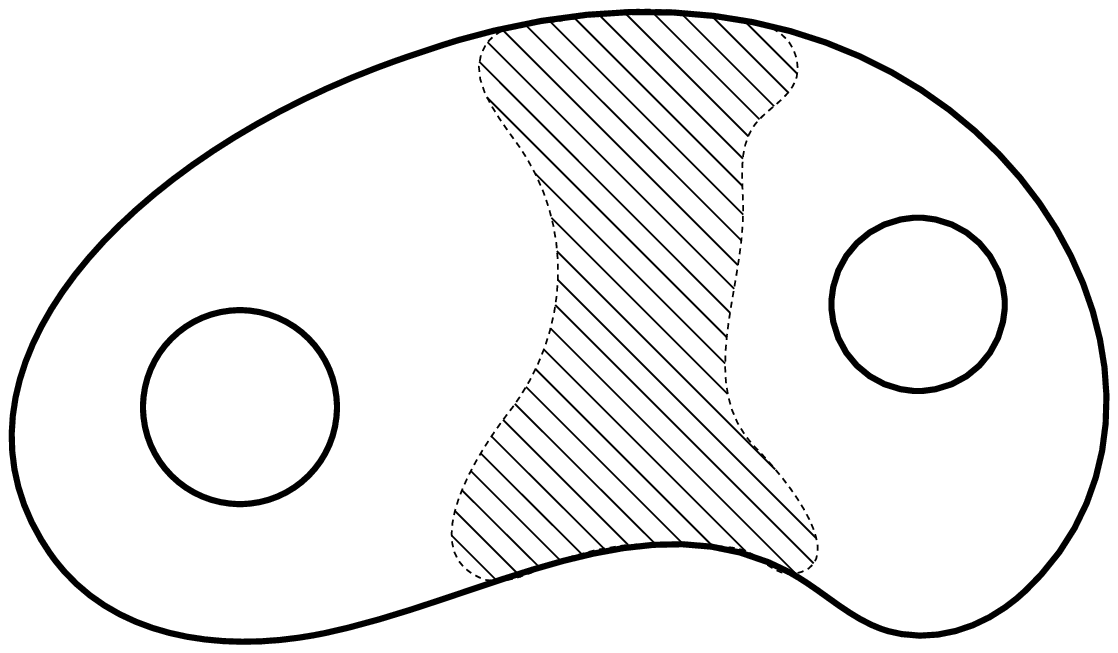}
 \includegraphics[scale=0.6]{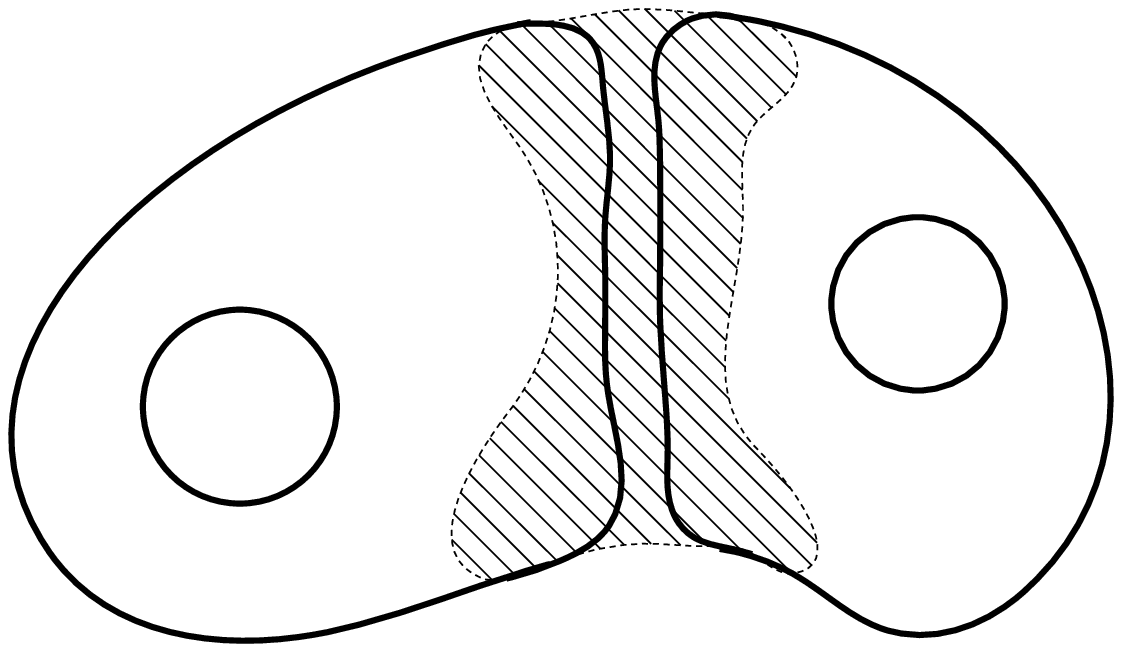}
\vglue5pt
\caption{Cutting a domain with holes into pieces.\label{fig1}}
\end{figure}

\begin{figure}
 \centering
 \includegraphics[scale=0.6]{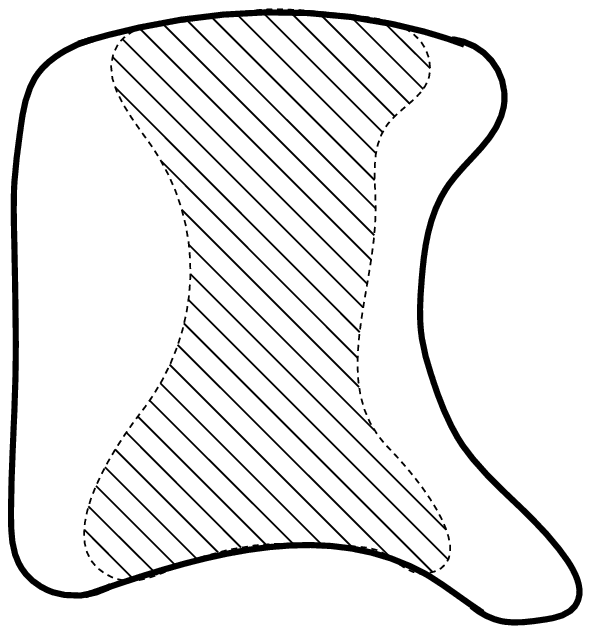}
 \includegraphics[scale=0.6]{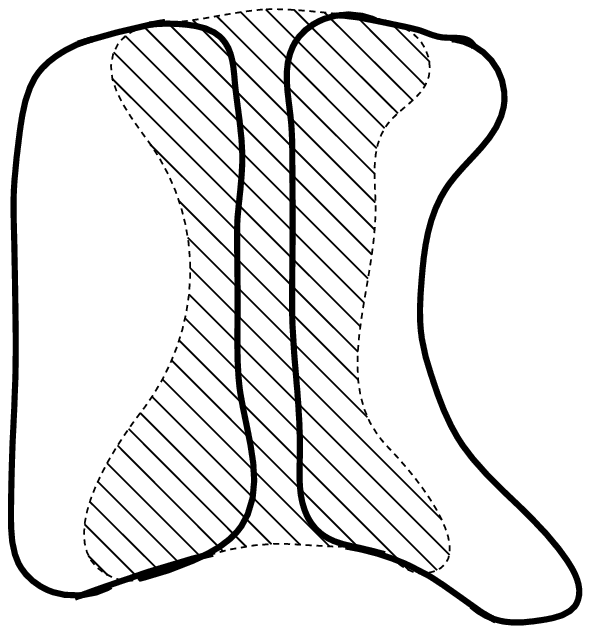}
\vglue5pt
\caption{Cutting a domain without holes into pieces.\label{fig2}}
\end{figure}

The localization principle permits one to split the domain with
holes into parts with fewer holes. Let us show this by example.
Figure~\ref{fig1}, left shows a domain $X$ with two holes. Let us
reduce the computation of the spectral flow of a family of Dirac
operators with local boundary conditions in $X$ to the
corresponding computation for domains with one hole. Let $N_1$ be
the set dashed in Fig.~\ref{fig1}, left, and let $N_2$ be the
complement to a small neighborhood of $N_1$. Let us change the
domain inside $N_1$ as shown in Fig.~\ref{fig1}, right, so that the
original domain becomes two domains with smooth boundary. The
function $B(x)$ occurring in the boundary conditions can be
extended by continuity as a nonvanishing real-valued function to
the newly arising boundary arcs inside $N_1$, because the sign of
$B(x)$ is the same on the entire outer boundary of the original
domain. Thus, the domain splits into two unrelated parts, and to
prove that the spectral flow of the family of Dirac operators in
the original domain is equal to the sum of spectral flows
corresponding to the two new domains, one should show that the
increment of the spectral flow under this modification of the
domain is zero. To this end, we use the localization method. Let us
change the original family by changing the domain in $N_2$ (so that
the resulting domain has the form shown in Fig.~\ref{fig2}, left)
and by extending $B(x)$ by continuity as a nonvanishing function to
the newly arising boundary arcs inside $N_2$. The spectral flow of
the new family is zero before as well as after the modification
shown in Fig.~\ref{fig2}, because the domains in this figure are
contractible and the gauge transformation $\mu$ in these domains is
homotopic to the identity transformation. Thus, cutting the domain
into pieces reduces the problem to the case of domains with one
hole, for which formula \eqref{e9} was proved in \cite{Pro1}.
(Needless to say, one can prove it directly with the use of Theorem
\ref{t1}, but there is no need to do this, and we omit the
corresponding computations for lack of space.)
\end{proof}

\begin{center}
\textbf{\large{5. CONCLUSIONS}}
\end{center}

Let us discuss possible physical realizations of nonzero spectral
flow. We start with the case of graphene. One has to keep in mind
that there are {\it two} Dirac electron subsystems in graphene (two
valleys) and, generally speaking,  scattering at the edges mixes
the two valleys \cite{Kats,17*}. This is not the case, however, if
an energy gap in the electron energy spectrum opens smoothly when
reaching the edge. At the chemical functionalization of the edges
this is, indeed, the case, since electronic structure is modified
in a sufficiently large region of space \cite{18*}. As a result,
intervalley scattering is negligible and we have the boundary
condition~\eqref{ei1} suggested first by Berry and Mondragon
\cite{19*}. A detailed microscopic derivation of the boundary
condition starting with a discrete lattice model has been done in
Ref. \cite{17*} (see also \cite[Chapter~5]{Kats} and references
therein). The sign of the constant $B$ is determined by the sign of
the mass term in the Dirac equation and is dependent on the
distribution of chemical groups along the edge. One can hope that
if we prepare graphene rings, then in some specimens the signs of
$B$ will be opposite at the internal and external edges of the
ring, which is necessary for nonzero spectral flow. However, it is
hard to reach in a controllable way.

Probably, topological insulators are more promising in this sense.
First, two-dimensional massless Dirac fermions are realized at the
surface of three- dimensional topological insulators, such as
$\operatorname{Bi}_2\operatorname{Se}_3$, only one Dirac cone
arising \cite{8*,9*}. To open the gap, one has to cover the surface
by a magnetic layer, the sign of the gap being determined by the
direction of magnetization \cite{8*,9*}. This opens a way to
manipulate the sign of the constant $B$. Second, two-dimensional
massless Dirac fermions can be realized in a layer of
$\operatorname{HgTe}$ confined between two layers of
$\operatorname{CdTe}$, at a certain critical thickness of the layer
\cite{8*,9*}. Recently, such an opportunity has been demonstrated
experimentally \cite{20*}. If the thickness of the layer varies
smoothly in space oscillating near the critical value, one can
reach both positive and negative values of $B$. Currently, this
opportunity to create nonzero spectral flow looks the most
promising.

\paragraph{Acknowledgments.} The second author's research was supported by
the Russian Foundation for Basic Research (grant no.~11-01-00973).
The authors are grateful to S.~Yu~Dobrokhotov and A.~I.~Shafarevich
for attention and useful discussion.

\end{document}